\documentclass[10pt]{article}
\usepackage{amsmath}
\usepackage{amssymb}
\usepackage{amsthm}
\usepackage{mathrsfs}
\usepackage[compress,noadjust]{cite}
\numberwithin{equation}{section}

\begin{document}
\title{Multilinear singular and fractional integral operators on weighted Morrey spaces}
\author{Hua Wang \\
Department of Mathematics,\\
Zhejiang University, Hangzhou 310027, P. R. China\\
E-mail address: wanghua@pku.edu.cn.\\
\and Wentan Yi \\
Department of Applied Mathematics,\\
Zhengzhou Information Science and Technology Institute,\\
Zhengzhou 450002, P. R. China\\
E-mail address: nlwt89@sina.com.}
\date{}
\maketitle

\begin{abstract}
In this paper, we will study the boundedness properties of multilinear Calder\'on--Zygmund operators and multilinear fractional integrals on products of weighted Morrey spaces with multiple weights.\\
MSC(2010): 42B20; 42B35\\
Keywords: Multilinear Calder\'on--Zygmund operators; multilinear fractional integrals; weighted Morrey spaces; multiple weights
\end{abstract}

\section{Introduction and main results}

Multinear Calder\'on--Zygmund theory is a natural generalization of the linear case. The initial work on the class of multilinear Calder\'on--Zygmund operators was done by Coifman and Meyer in \cite{coifman}, and was later systematically studied by Grafakos and Torres in \cite{grafakos2,grafakos3,grafakos4}. Let $\mathbb R^n$ be the $n$-dimensional Euclidean space and $(\mathbb R^n)^m=\mathbb R^n\times\cdots\times\mathbb R^n$ be the $m$-fold product space ($m\in\mathbb N$). We denote by $\mathscr S(\mathbb R^n)$ the space of all Schwartz functions on $\mathbb R^n$ and by $\mathscr S'(\mathbb R^n)$ its dual space, the set of all tempered distributions on $\mathbb R^n$. Let $m\ge2$ and $T$ be an $m$-linear operator initially defined on the $m$-fold product of Schwartz spaces and taking values into the space of tempered distributions,
\begin{equation*}
T:\mathscr S(\mathbb R^n)\times\cdots\times\mathscr S(\mathbb R^n)\to\mathscr S'(\mathbb R^n).
\end{equation*}
Following \cite{grafakos2}, for given $\vec{f}=(f_1,\ldots,f_m)$, we say that $T$ is an $m$-linear Calder\'on--Zygmund operator if for some $q_1,\ldots,q_m\in[1,\infty)$ and $q\in(0,\infty)$ with $1/q=\sum_{k=1}^m 1/{q_k}$, it extends to a bounded multilinear operator from $L^{q_1}(\mathbb R^n)\times\cdots\times L^{q_m}(\mathbb R^n)$ into $L^q(\mathbb R^n)$, and if there exists a kernel function $K(x,y_1,\ldots,y_m)$ in the class $m$-$CZK(A,\varepsilon)$, defined away from the diagonal $x=y_1=\cdots=y_m$ in $(\mathbb R^n)^{m+1}$ such that
\begin{equation}
T(\vec{f})(x)=T(f_1,\ldots,f_m)(x)=\int_{(\mathbb R^n)^m}K(x,y_1,\ldots,y_m)
f_1(y_1)\cdots f_m(y_m)\,dy_1\cdots dy_m,
\end{equation}
whenever $f_1,\ldots,f_m\in \mathscr S(\mathbb R^n)$ and $x\notin \cap_{k=1}^m$ supp\,$f_k$. We say that $K(x,y_1,\ldots,y_m)$ is a kernel in the class $m$-$CZK(A,\varepsilon)$, if it satisfies the size condition
\begin{equation*}
\big|K(x,y_1,\ldots,y_m)\big|\le\frac{A}{(|x-y_1|+\cdots+|x-y_m|)^{mn}},
\end{equation*}
for some $A>0$ and all $(x,y_1,\ldots,y_m)\in(\mathbb R^n)^{m+1}$ with $x\neq y_k$ for some $1\le k\le m$. Moreover, for some $\varepsilon>0$, it satisfies the regularity condition that
\begin{equation*}
\big|K(x,y_1,\ldots,y_m)-K(x',y_1,\ldots,y_m)\big|\le
\frac{A\cdot|x-x'|^\varepsilon}{(|x-y_1|+\cdots+|x-y_m|)^{mn+\varepsilon}}
\end{equation*}
whenever $|x-x'|\le\frac12 \max_{1\le k\le m}|x-y_k|$, and also that for each fixed $k$ with $1\le k\le m$,
\begin{equation*}
\big|K(x,y_1,\ldots,y_k,\ldots,y_m)-K(x,y_1,\ldots,y'_k,\ldots,y_m)\big|\le
\frac{A\cdot|y_k-y'_k|^\varepsilon}{(|x-y_1|+\cdots+|x-y_m|)^{mn+\varepsilon}}
\end{equation*}
whenever $|y_k-y'_k|\le\frac12 \max_{1\le i\le m}|x-y_i|$. In recent years, many authors have been interested in studying the boundedness of these operators on function spaces, see e.g.\cite{grafakos5,hu,li1,li2}. In 2009, the weighted strong and weak type estimates of multilinear Calder\'on--Zygmund singular integral operators were established in \cite{lerner} by Lerner et al. New more refined multilinear maximal function was defined and used in \cite{lerner} to characterize the class of multiple $A_{\vec{P}}$ weights.

\newtheorem*{thma}{Theorem A}

\begin{thma}[\cite{lerner}]
Let $m\ge2$ and $T$ be an $m$-linear Calder\'on--Zygmund operator. If $p_1,\ldots,p_m\in(1,\infty)$ and $p\in(0,\infty)$ with $1/p=\sum_{k=1}^m 1/{p_k}$, and $\vec{w}=(w_1,\ldots,w_m)$ satisfy the $A_{\vec{P}}$ condition, then there exists a constant $C>0$ independent of $\vec{f}=(f_1,\ldots,f_m)$ such that
\begin{equation*}
\big\|T(\vec{f})\big\|_{L^p(\nu_{\vec{w}})}\le C\prod_{i=1}^m\big\|f_i\big\|_{L^{p_i}(w_i)},
\end{equation*}
where $\nu_{\vec{w}}=\prod_{i=1}^m w_i^{p/{p_i}}$.
\end{thma}

\newtheorem*{thmb}{Theorem B}

\begin{thmb}[\cite{lerner}]
Let $m\ge2$ and $T$ be an $m$-linear Calder\'on--Zygmund operator. If $p_1,\ldots,p_m\in[1,\infty)$, $\min\{p_1,\ldots,p_m\}=1$ and $p\in(0,\infty)$ with $1/p=\sum_{k=1}^m 1/{p_k}$, and $\vec{w}=(w_1,\ldots,w_m)$ satisfy the $A_{\vec{P}}$ condition, then there exists a constant $C>0$ independent of $\vec{f}=(f_1,\ldots,f_m)$ such that
\begin{equation*}
\big\|T(\vec{f})\big\|_{WL^p(\nu_{\vec{w}})}\le C\prod_{i=1}^m\big\|f_i\big\|_{L^{p_i}(w_i)},
\end{equation*}
where $\nu_{\vec{w}}=\prod_{i=1}^m w_i^{p/{p_i}}$.
\end{thmb}

Let $m\ge2$ and $0<\alpha<mn$. For given $\vec{f}=(f_1,\ldots,f_m)$, the $m$-linear fractional integral operator is defined by
\begin{equation}
I_{\alpha}(\vec{f})(x)=I_{\alpha}(f_1,\ldots,f_m)(x)=
\int_{(\mathbb R^n)^m}\frac{f_1(y_1)\cdots f_m(y_m)}{|(x-y_1,\ldots,x-y_m)|^{mn-\alpha}}\,dy_1\cdots dy_m.
\end{equation}
For the boundedness properties of multilinear fractional integrals on various function spaces, we refer the reader to \cite{grafakos1,Iida1,Iida2,Iida3,kenig,pradolini,tang}. In 2009, Moen \cite{moen} considered the weighted norm inequalities for multilinear fractional integral operators and constructed the class of multiple $A_{\vec{P},q}$ weights (see also \cite{chen}).

\newtheorem*{thmc}{Theorem C}
\begin{thmc}[\cite{chen,moen}]
Let $m\ge2$, $0<\alpha<mn$ and $I_\alpha$ be an $m$-linear fractional integral operator. If $p_1,\ldots,p_m\in(1,\infty)$, $1/p=\sum_{k=1}^m 1/{p_k}$ and $1/q=1/p-\alpha/n$, and $\vec{w}=(w_1,\ldots,w_m)$ satisfy the $A_{\vec{P},q}$ condition, then there exists a constant $C>0$ independent of $\vec{f}=(f_1,\ldots,f_m)$ such that
\begin{equation*}
\big\|I_\alpha(\vec{f})\big\|_{L^q((\nu_{\vec{w}})^q)}\le C\prod_{i=1}^m\big\|f_i\big\|_{L^{p_i}(w^{p_i}_i)},
\end{equation*}
where $\nu_{\vec{w}}=\prod_{i=1}^m w_i$.
\end{thmc}

\newtheorem*{thmd}{Theorem D}
\begin{thmd}[\cite{chen,moen}]
Let $m\ge2$, $0<\alpha<mn$ and $I_\alpha$ be an $m$-linear fractional integral operator. If $p_1,\ldots,p_m\in[1,\infty)$, $\min\{p_1,\ldots,p_m\}=1$, $1/p=\sum_{k=1}^m 1/{p_k}$ and $1/q=1/p-\alpha/n$, and $\vec{w}=(w_1,\ldots,w_m)$ satisfy the $A_{\vec{P},q}$ condition, then there exists a constant $C>0$ independent of $\vec{f}=(f_1,\ldots,f_m)$ such that
\begin{equation*}
\big\|I_\alpha(\vec{f})\big\|_{WL^q((\nu_{\vec{w}})^q)}\le C\prod_{i=1}^m\big\|f_i\big\|_{L^{p_i}(w^{p_i}_i)},
\end{equation*}
where $\nu_{\vec{w}}=\prod_{i=1}^m w_i$.
\end{thmd}

On the other hand, the classical Morrey spaces $\mathcal L^{p,\lambda}$ were originally introduced by Morrey in \cite{morrey} to study the local behavior of solutions to second order elliptic partial differential equations. For the boundedness of the Hardy--Littlewood maximal operator, the fractional integral operator and the Calder\'on--Zygmund singular integral operator on these spaces, we refer the reader to \cite{adams,chiarenza,peetre}. For the properties and applications of classical Morrey spaces, one can see \cite{fan,fazio1,fazio2} and the references therein.

In 2009, Komori and Shirai \cite{komori} first defined the weighted Morrey spaces $L^{p,\kappa}(w)$ which could be
viewed as an extension of weighted Lebesgue spaces, and studied the boundedness of the above classical operators in Harmonic Analysis on these weighted spaces. Recently, in \cite{wang1,wang2,wang3,wang4,wang5,wang6,wang7,wang8}, we have established the continuity properties of some other operators and their commutators on the weighted Morrey spaces $L^{p,\kappa}(w)$.

The main purpose of this paper is to establish the boundedness properties of multilinear Calder\'on--Zygmund operators and multilinear fractional integrals on products of weighted Morrey spaces with multiple weights. We now formulate our main results as follows.

\newtheorem{theorem}{Theorem}[section]

\begin{theorem}
Let $m\ge2$ and $T$ be an $m$-linear Calder\'on--Zygmund operator. If $p_1,\ldots,p_m\in(1,\infty)$ and $p\in(0,\infty)$ with $1/p=\sum_{k=1}^m 1/{p_k}$, and $\vec{w}=(w_1,\ldots,w_m)\in A_{\vec{P}}$ with $w_1,\ldots,w_m\in A_\infty$, then for any $0<\kappa<1$, there exists a
constant $C>0$ independent of $\vec{f}=(f_1,\ldots,f_m)$ such that
\begin{equation*}
\big\|T(\vec{f})\big\|_{L^{p,\kappa}(\nu_{\vec{w}})}\le C\prod_{i=1}^m\big\|f_i\big\|_{L^{p_i,\kappa}(w_i)},
\end{equation*}
where $\nu_{\vec{w}}=\prod_{i=1}^m w_i^{p/{p_i}}$.
\end{theorem}

\begin{theorem}
Let $m\ge2$ and $T$ be an $m$-linear Calder\'on--Zygmund operator. If $p_1,\ldots,p_m\in[1,\infty)$, $\min\{p_1,\ldots,p_m\}=1$ and $p\in(0,\infty)$ with $1/p=\sum_{k=1}^m 1/{p_k}$, and $\vec{w}=(w_1,\ldots,w_m)\in A_{\vec{P}}$ with $w_1,\ldots,w_m\in A_\infty$, then for any $0<\kappa<1$, there exists a
constant $C>0$ independent of $\vec{f}=(f_1,\ldots,f_m)$ such that
\begin{equation*}
\big\|T(\vec{f})\big\|_{WL^{p,\kappa}(\nu_{\vec{w}})}\le C\prod_{i=1}^m\big\|f_i\big\|_{L^{p_i,\kappa}(w_i)},
\end{equation*}
where $\nu_{\vec{w}}=\prod_{i=1}^m w_i^{p/{p_i}}$.
\end{theorem}

\begin{theorem}
Let $m\ge2$, $0<\alpha<mn$ and $I_\alpha$ be an $m$-linear fractional integral operator. If $p_1,\ldots,p_m\in(1,\infty)$, $1/p=\sum_{k=1}^m 1/{p_k}$, $1/{q_k}=1/{p_k}-\alpha/{mn}$ and $1/q=\sum_{k=1}^m 1/{q_k}=1/p-\alpha/n$, and $\vec{w}=(w_1,\ldots,w_m)\in A_{\vec{P},q}$ with $w^{q_1}_1,\ldots,w^{q_m}_m\in A_\infty$, then for any $0<\kappa<p/q$, there exists a
constant $C>0$ independent of $\vec{f}=(f_1,\ldots,f_m)$ such that
\begin{equation*}
\big\|I_\alpha(\vec{f})\big\|_{L^{q,{\kappa q}/p}((\nu_{\vec{w}})^q)}\le C\prod_{i=1}^m
\big\|f_i\big\|_{L^{p_i,{\kappa p_iq}/{pq_i}}(w^{p_i}_i,w^{q_i}_i)},
\end{equation*}
where $\nu_{\vec{w}}=\prod_{i=1}^m w_i$.
\end{theorem}

\begin{theorem}
Let $m\ge2$, $0<\alpha<mn$ and $I_\alpha$ be an $m$-linear fractional integral operator. If $p_1,\ldots,p_m\in[1,\infty)$, $\min\{p_1,\ldots,p_m\}=1$, $1/p=\sum_{k=1}^m 1/{p_k}$, $1/{q_k}=1/{p_k}-\alpha/{mn}$ and $1/q=\sum_{k=1}^m 1/{q_k}=1/p-\alpha/n$, and $\vec{w}=(w_1,\ldots,w_m)\in A_{\vec{P},q}$ with $w^{q_1}_1,\ldots,w^{q_m}_m\in A_\infty$, then for any $0<\kappa<p/q$, there exists a constant $C>0$ independent of $\vec{f}=(f_1,\ldots,f_m)$ such that
\begin{equation*}
\big\|I_\alpha(\vec{f})\big\|_{WL^{q,{\kappa q}/p}((\nu_{\vec{w}})^q)}\le C\prod_{i=1}^m
\big\|f_i\big\|_{L^{p_i,{\kappa p_iq}/{pq_i}}(w^{p_i}_i,w^{q_i}_i)},
\end{equation*}
where $\nu_{\vec{w}}=\prod_{i=1}^m w_i$.
\end{theorem}

\section{Notations and definitions}

The classical $A_p$ weight theory was first introduced by Muckenhoupt in the study of weighted $L^p$ boundedness of Hardy--Littlewood maximal functions in \cite{muckenhoupt1}. A weight $w$ is a nonnegative, locally integrable function on $\mathbb R^n$, $B=B(x_0,r_B)$ denotes the ball with the center $x_0$ and radius $r_B$. For $1<p<\infty$, a weight function $w$ is said to belong to $A_p$, if there is a constant $C>0$ such that for every ball $B\subseteq \mathbb R^n$,
\begin{equation}
\left(\frac1{|B|}\int_B w(x)\,dx\right)\left(\frac1{|B|}\int_B w(x)^{-1/{(p-1)}}\,dx\right)^{p-1}\le C,
\end{equation}
where $|B|$ denotes the Lebesgue measure of $B$. For the case $p=1$, $w\in A_1$, if there is a constant $C>0$ such that for every ball $B\subseteq \mathbb R^n$,
\begin{equation}
\frac1{|B|}\int_B w(x)\,dx\le C\cdot\underset{x\in B}{\mbox{ess\,inf}}\;w(x).
\end{equation}
A weight function $w\in A_\infty$ if it satisfies the $A_p$ condition for some $1<p<\infty$.
We also need another weight class $A_{p,q}$ introduced by Muckenhoupt and Wheeden in \cite{muckenhoupt2}. A weight function $w$ belongs to $A_{p,q}$ for $1<p<q<\infty$ if there is a constant $C>0$ such that for every ball $B\subseteq \mathbb R^n$,
\begin{equation}
\left(\frac{1}{|B|}\int_B w(x)^q\,dx\right)^{1/q}\left(\frac{1}{|B|}\int_B w(x)^{-p'}\,dx\right)^{1/{p'}}\le C.
\end{equation}
When $p=1$, $w$ is in the class $A_{1,q}$ with $1<q<\infty$ if there is a constant $C>0$ such that for every ball $B\subseteq \mathbb R^n$,
\begin{equation}
\left(\frac{1}{|B|}\int_B w(x)^q\,dx\right)^{1/q}\bigg(\underset{x\in B}{\mbox{ess\,sup}}\,\frac{1}{w(x)}\bigg)\le C.
\end{equation}

Now let us recall the definitions of multiple weights. For $m$ exponents $p_1,\ldots,p_m$, we will write $\vec{P}$ for the vector $\vec{P}=(p_1,\ldots,p_m)$. Let $p_1,\ldots,p_m\in[1,\infty)$ and $p\in(0,\infty)$ with $1/p=\sum_{k=1}^m 1/{p_k}$. Given $\vec{w}=(w_1,\ldots,w_m)$, set $\nu_{\vec{w}}=\prod_{i=1}^m w_i^{p/{p_i}}$. We say that $\vec{w}$ satisfies the $A_{\vec{P}}$ condition if it satisfies
\begin{equation}
\sup_B\left(\frac{1}{|B|}\int_B \nu_{\vec{w}}(x)\,dx\right)^{1/p}\prod_{i=1}^m\left(\frac{1}{|B|}\int_B w_i(x)^{1-p'_i}\,dx\right)^{1/{p'_i}}<\infty.
\end{equation}
When $p_i=1$, $\big(\frac{1}{|B|}\int_B w_i(x)^{1-p'_i}\,dx\big)^{1/{p'_i}}$ is understood as $\big(\inf_{x\in B}w_i(x)\big)^{-1}$.

Let $p_1,\ldots,p_m\in[1,\infty)$, $1/p=\sum_{k=1}^m 1/{p_k}$ and $q>0$. Given $\vec{w}=(w_1,\ldots,w_m)$, set $\nu_{\vec{w}}=\prod_{i=1}^m w_i$. We say that $\vec{w}$ satisfies the $A_{\vec{P},q}$ condition if it satisfies
\begin{equation}
\sup_B\left(\frac{1}{|B|}\int_B \nu_{\vec{w}}(x)^q\,dx\right)^{1/q}\prod_{i=1}^m\left(\frac{1}{|B|}\int_B w_i(x)^{-p'_i}\,dx\right)^{1/{p'_i}}<\infty.
\end{equation}
When $p_i=1$, $\big(\frac{1}{|B|}\int_B w_i(x)^{-p'_i}\,dx\big)^{1/{p'_i}}$ is understood as $\big(\inf_{x\in B}w_i(x)\big)^{-1}$.

Given a ball $B$ and $\lambda>0$, $\lambda B$ denotes the ball with the same center as $B$ whose radius is $\lambda$ times that of $B$. For a given weight function $w$ and a measurable set $E$, we also denote the Lebesgue measure of $E$ by $|E|$ and the weighted measure of $E$ by $w(E)$, where $w(E)=\int_E w(x)\,dx$.

\newtheorem{lemma}[theorem]{Lemma}

\begin{lemma}[\cite{garcia}]
Let $w\in A_p$ with $1\le p<\infty$. Then, for any ball $B$, there exists an absolute constant $C>0$ such that
\begin{equation*}
w(2B)\le C\,w(B).
\end{equation*}
\end{lemma}

\begin{lemma}[\cite{duoand}]
Let $w\in A_\infty$. Then for all balls $B\subseteq \mathbb R^n$, the following reverse Jensen inequality holds.
\begin{equation*}
\int_B w(x)\,dx\le C|B|\cdot\exp\left(\frac{1}{|B|}\int_B\log w(x)\,dx\right).
\end{equation*}
\end{lemma}

\begin{lemma}[\cite{garcia}]
Let $w\in A_\infty$. Then for all balls $B$ and all measurable subsets $E$ of $B$, there exists $\delta>0$ such that
\begin{equation*}
\frac{w(E)}{w(B)}\le C\left(\frac{|E|}{|B|}\right)^\delta.
\end{equation*}
\end{lemma}

\begin{lemma}[\cite{lerner}]
Let $p_1,\ldots,p_m\in[1,\infty)$ and $1/p=\sum_{k=1}^m 1/{p_k}$. Then $\vec{w}=(w_1,\ldots,w_m)\in A_{\vec{P}}$ if and only if
\begin{equation*}\left\{
\begin{aligned}
&\nu_{\vec{w}}\in A_{mp},\\
&w_i^{1-p'_i}\in A_{mp'_i},\quad i=1,\ldots,m,
\end{aligned}\right.
\end{equation*}
where $\nu_{\vec{w}}=\prod_{i=1}^m w_i^{p/{p_i}}$ and the condition $w_i^{1-p'_i}\in A_{mp'_i}$ in the case $p_i=1$ is understood as $w_i^{1/m}\in A_1$.
\end{lemma}

\begin{lemma}[\cite{chen,moen}]
Let $0<\alpha<mn$, $p_1,\ldots,p_m\in[1,\infty)$, $1/p=\sum_{k=1}^m 1/{p_k}$ and $1/q=1/p-\alpha/n$. Then $\vec{w}=(w_1,\ldots,w_m)\in A_{\vec{P},q}$ if and only if
\begin{equation*}\left\{
\begin{aligned}
&(\nu_{\vec{w}})^q\in A_{mq},\\
&w_i^{-p'_i}\in A_{mp'_i},\quad i=1,\ldots,m,
\end{aligned}\right.
\end{equation*}
where $\nu_{\vec{w}}=\prod_{i=1}^m w_i$.
\end{lemma}

Given a weight function $w$ on $\mathbb R^n$, for $0<p<\infty$, the weighted Lebesgue space $L^p(w)$ defined as the set of all functions $f$ such that
\begin{equation}
\big\|f\big\|_{L^p(w)}=\bigg(\int_{\mathbb R^n}|f(x)|^pw(x)\,dx\bigg)^{1/p}<\infty.
\end{equation}
We also denote by $WL^p(w)$ the weighted weak space consisting of all measurable functions $f$ such that
\begin{equation}
\big\|f\big\|_{WL^p(w)}=\sup_{\lambda>0}\lambda\cdot w\big(\big\{x\in\mathbb R^n:|f(x)|>\lambda \big\}\big)^{1/p}<\infty.
\end{equation}

In 2009, Komori and Shirai \cite{komori} first defined the weighted Morrey spaces $L^{p,\kappa}(w)$ for $1\le p<\infty$. In order to deal with the multilinear case $m\ge2$, we shall define $L^{p,\kappa}(w)$ for all $0<p<\infty$.

\newtheorem{defn}[theorem]{Definition}

\begin{defn}
Let $0<p<\infty$, $0<\kappa<1$ and $w$ be a weight function on $\mathbb R^n$. Then the weighted Morrey space is defined by
\begin{equation*}
L^{p,\kappa}(w)=\big\{f\in L^p_{loc}(w):\big\|f\big\|_{L^{p,\kappa}(w)}<\infty\big\},
\end{equation*}
where
\begin{equation*}
\big\|f\big\|_{L^{p,\kappa}(w)}=\sup_B\left(\frac{1}{w(B)^\kappa}\int_B|f(x)|^pw(x)\,dx\right)^{1/p}
\end{equation*}
and the supremum is taken over all balls $B$ in $\mathbb R^n$.
\end{defn}

\begin{defn}
Let $0<p<\infty$, $0<\kappa<1$ and $w$ be a weight function on $\mathbb R^n$. Then the weighted weak Morrey space is defined by
\begin{equation*}
WL^{p,\kappa}(w)=\big\{f \;\mbox{measurable}:\big\|f\big\|_{WL^{p,\kappa}(w)}<\infty\big\},
\end{equation*}
where
\begin{equation*}
\big\|f\big\|_{WL^{p,\kappa}(w)}=\sup_B\sup_{\lambda>0}\frac{1}{w(B)^{\kappa/p}}\lambda\cdot
w\big(\big\{x\in B:|f(x)|>\lambda \big\}\big)^{1/p}.
\end{equation*}
\end{defn}

Furthermore, in order to deal with the fractional order case, we need to consider the weighted Morrey spaces with two weights.

\begin{defn}
Let $0<p<\infty$ and $0<\kappa<1$. Then for two weights $u$ and $v$, the weighted Morrey space is defined by
\begin{equation*}
L^{p,\kappa}(u,v)=\big\{f\in L^p_{loc}(u):\big\|f\big\|_{L^{p,\kappa}(u,v)}<\infty\big\},
\end{equation*}
where
\begin{equation*}
\big\|f\big\|_{L^{p,\kappa}(u,v)}=\sup_{B}\left(\frac{1}{v(B)^{\kappa}}\int_B|f(x)|^pu(x)\,dx\right)^{1/p}.
\end{equation*}
\end{defn}

Throughout this article, we will use $C$ to denote a positive constant, which is independent of the main parameters
and not necessarily the same at each occurrence. Moreover, we will denote the conjugate exponent of $p>1$ by $p'=p/{(p-1)}$.

\section{Proofs of Theorems 1.1 and 1.2}

Before proving the main theorems of this section, we need to establish the following lemma.

\begin{lemma}
Let $m\ge2$, $p_1,\ldots,p_m\in[1,\infty)$ and $p\in(0,\infty)$ with $1/p=\sum_{k=1}^m 1/{p_k}$. Assume that $w_1,\ldots,w_m\in A_\infty$ and $\nu_{\vec{w}}=\prod_{i=1}^m w_i^{p/{p_i}}$, then for any ball $B$, there exists a constant $C>0$ such that
\begin{equation*}
\prod_{i=1}^m\left(\int_B w_i(x)\,dx\right)^{p/{p_i}}\le C\int_B \nu_{\vec{w}}(x)\,dx.
\end{equation*}
\end{lemma}

\begin{proof}
Since $w_1,\ldots,w_m\in A_\infty$, then by using Lemma 2.2, we have
\begin{equation*}
\begin{split}
\prod_{i=1}^m \left(\int_B w_i(x)\,dx\right)^{p/{p_i}}&\le C\prod_{i=1}^m
\left(|B|\cdot\exp\bigg(\frac{1}{|B|}\int_B \log w_i(x)\,dx\bigg)\right)^{p/{p_i}}\\
&= C\prod_{i=1}^m \left(|B|^{p/{p_i}}\cdot\exp\bigg(\frac{1}{|B|}\int_B\log w_i(x)^{p/{p_i}}\,dx\bigg)\right)\\
&= C\cdot\big(|B|\big)^{\sum_{i=1}^m p/{p_i}}\cdot\exp\left(\sum_{i=1}^m\frac{1}{|B|}\int_B\log w_i(x)^{p/{p_i}}\,dx\right).
\end{split}
\end{equation*}
Note that $\sum_{i=1}^m p/{p_i}=1$ and $\nu_{\vec{w}}(x)=\prod_{i=1}^m w_i(x)^{p/{p_i}}$. Then by Jensen inequality, we obtain
\begin{equation*}
\begin{split}
\prod_{i=1}^m \left(\int_B w_i(x)\,dx\right)^{p/{p_i}}&\le C\cdot|B|\cdot\exp\left(\frac{1}{|B|}\int_B\log \nu_{\vec{w}}(x)\,dx\right)\\
&\le C\int_B \nu_{\vec{w}}(x)\,dx.
\end{split}
\end{equation*}
We are done.
\end{proof}

\begin{proof}[Proof of Theorem 1.1]
For any ball $B=B(x_0,r_B)\subseteq\mathbb R^n$ and let $f_i=f^0_i+f^{\infty}_i$, where $f^0_i=f_i\chi_{2B}$, $i=1,\ldots,m$ and $\chi_{2B}$ denotes the characteristic function of $2B$. Then we write
\begin{equation*}
\begin{split}
\prod_{i=1}^m f_i(y_i)&=\prod_{i=1}^m\Big(f^0_i(y_i)+f^{\infty}_i(y_i)\Big)\\
&=\sum_{\alpha_1,\ldots,\alpha_m\in\{0,\infty\}}f^{\alpha_1}_1(y_1)\cdots f^{\alpha_m}_m(y_m)\\
&=\prod_{i=1}^m f^0_i(y_i)+\sum\nolimits^\prime f^{\alpha_1}_1(y_1)\cdots f^{\alpha_m}_m(y_m),
\end{split}
\end{equation*}
where each term of $\sum^\prime$ contains at least one $\alpha_i\neq0$. Since $T$ is an $m$-linear operator, then we have
\begin{equation*}
\begin{split}
&\frac{1}{\nu_{\vec{w}}(B)^{\kappa/p}}\left(\int_B \big|T(f_1,\ldots,f_m)(x)\big|^p\nu_{\vec{w}}(x)\,dx\right)^{1/p}\\
\le& \frac{1}{\nu_{\vec{w}}(B)^{\kappa/p}}\left(\int_B \big|T(f^0_1,\ldots,f^0_m)(x)\big|^p\nu_{\vec{w}}(x)\,dx\right)^{1/p}\\
&+\sum\nolimits^\prime\frac{1}{\nu_{\vec{w}}(B)^{\kappa/p}}\left(\int_B \big|T(f^{\alpha_1}_1,\ldots,f^{\alpha_m}_m)(x)\big|^p\nu_{\vec{w}}(x)\,dx\right)^{1/p}\\
=& I^0+\sum\nolimits^\prime I^{\alpha_1,\ldots,\alpha_m}.
\end{split}
\end{equation*}
In view of Lemma 2.4, we have that $\nu_{\vec{w}}\in A_{mp}$. Applying Theorem A, Lemma 3.1 and Lemma 2.1, we get
\begin{equation*}
\begin{split}
I^0&\le C\cdot
\frac{1}{\nu_{\vec{w}}(B)^{\kappa/p}}\prod_{i=1}^m\left(\int_{2B}|f_i(x)|^{p_i}w_i(x)\,dx\right)^{1/{p_i}}\\
&\le C\prod_{i=1}^m\big\|f_i\big\|_{L^{p_i,\kappa}(w_i)}\cdot
\frac{\prod_{i=1}^m w_i(2B)^{\kappa/{p_i}}}{\nu_{\vec{w}}(B)^{\kappa/p}}\\
&\le C\prod_{i=1}^m\big\|f_i\big\|_{L^{p_i,\kappa}(w_i)}\cdot
\frac{\nu_{\vec{w}}(2B)^{\kappa/p}}{\nu_{\vec{w}}(B)^{\kappa/p}}\\
&\le C\prod_{i=1}^m \big\|f_i\big\|_{L^{p_i,\kappa}(w_i)}.
\end{split}
\end{equation*}
For the other terms, let us first consider the case when $\alpha_1=\cdots=\alpha_m=\infty$. By the size condition, for any $x\in B$, we obtain
\begin{align}
\big|T(f^\infty_1,\ldots,f^\infty_m)(x)\big|&\le C\int_{(\mathbb R^n)^m\backslash(2B)^m}
\frac{|f_1(y_1)\cdots f_m(y_m)|}{(|x-y_1|+\cdots+|x-y_m|)^{mn}}dy_1\cdots dy_m\notag\\
&\le C\sum_{j=1}^\infty\int_{(2^{j+1}B)^m\backslash(2^{j}B)^m}
\frac{|f_1(y_1)\cdots f_m(y_m)|}{(|x-y_1|+\cdots+|x-y_m|)^{mn}}dy_1\cdots dy_m\notag\\
&\le C\sum_{j=1}^\infty\prod_{i=1}^m\int_{2^{j+1}B\backslash 2^{j}B}\frac{|f_i(y_i)|}{|x-y_i|^n}dy_i\notag\\
&\le C\sum_{j=1}^\infty\prod_{i=1}^m\frac{1}{|2^{j+1}B|}\int_{2^{j+1}B}\big|f_i(y_i)\big|\,dy_i,
\end{align}
where we have used the notation $E^m=E\times\cdots\times E$. Furthermore, by using H\"older's inequality, the multiple $A_{\vec{P}}$ condition and Lemma 3.1, we deduce that
\begin{equation*}
\begin{split}
\big|T(f^\infty_1,\ldots,f^\infty_m)(x)\big|&\le C\sum_{j=1}^\infty\prod_{i=1}^m
\frac{1}{|2^{j+1}B|}\left(\int_{2^{j+1}B}\big|f_i(y_i)\big|^{p_i}w_i(y_i)\,dy_i\right)^{1/{p_i}}
\left(\int_{2^{j+1}B}w_i(y_i)^{1-p'_i}\,dy_i\right)^{1/{p'_i}}\\
&\le C\sum_{j=1}^\infty\frac{1}{|2^{j+1}B|^m}\cdot
\frac{|2^{j+1}B|^{\frac 1p+\sum_{i=1}^m(1-\frac 1{p_i})}}{\nu_{\vec{w}}(2^{j+1}B)^{1/p}}
\prod_{i=1}^m
\left(\big\|f_i\big\|_{L^{p_i,\kappa}(w_i)}w_i\big(2^{j+1}B\big)^{\kappa/{p_i}}\right)\\
&\le C\prod_{i=1}^m \big\|f_i\big\|_{L^{p_i,\kappa}(w_i)}\cdot\sum_{j=1}^\infty
\left(\frac{\prod_{i=1}^m w_i(2^{j+1}B)^{\kappa/{p_i}}}{\nu_{\vec{w}}(2^{j+1}B)^{1/p}}\right)\\
&\le C\prod_{i=1}^m \big\|f_i\big\|_{L^{p_i,\kappa}(w_i)}\cdot
\sum_{j=1}^\infty\nu_{\vec{w}}\big(2^{j+1}B\big)^{{(\kappa-1)}/p}.
\end{split}
\end{equation*}
Since $\nu_{\vec{w}}\in A_{mp}\subset A_\infty$, then it follows directly from Lemma 2.3 that
\begin{equation}
\frac{\nu_{\vec{w}}(B)}{\nu_{\vec{w}}(2^{j+1}B)}\le C\left(\frac{|B|}{|2^{j+1}B|}\right)^\delta.
\end{equation}
Hence
\begin{equation*}
\begin{split}
I^{\infty,\ldots,\infty}&\le\nu_{\vec{w}}(B)^{{(1-\kappa)}/p}\big|T(f^\infty_1,\ldots,f^\infty_m)(x)\big|\\
&\le C\prod_{i=1}^m \big\|f_i\big\|_{L^{p_i,\kappa}(w_i)}\cdot\sum_{j=1}^\infty
\frac{\nu_{\vec{w}}(B)^{{(1-\kappa)}/p}}{\nu_{\vec{w}}(2^{j+1}B)^{{(1-\kappa)}/p}}\\
&\le C\prod_{i=1}^m \big\|f_i\big\|_{L^{p_i,\kappa}(w_i)}\cdot\sum_{j=1}^\infty
\left(\frac{|B|}{|2^{j+1}B|}\right)^{{\delta(1-\kappa)}/p}\\
&\le C\prod_{i=1}^m \big\|f_i\big\|_{L^{p_i,\kappa}(w_i)},
\end{split}
\end{equation*}
where the last inequality holds since $0<\kappa<1$ and $\delta>0$. We now consider the case where exactly $\ell$ of the $\alpha_i$ are $\infty$ for some $1\le\ell<m$. We
only give the arguments for one of these cases. The rest are similar and can easily be obtained from the arguments below by permuting the indices. Using the size condition again, we deduce that for any $x\in B$,
\begin{align}
\big|T(f^\infty_1,\ldots,f^\infty_\ell,f^0_{\ell+1},\ldots,f^0_m)(x)\big|&\le
C\int_{(\mathbb R^n)^{\ell}\backslash(2B)^{\ell}}\int_{(2B)^{m-\ell}}\frac{|f_1(y_1)\cdots f_m(y_m)|}{(|x-y_1|+\cdots+|x-y_m|)^{mn}}dy_1\cdots dy_m\notag\\
&\le C\prod_{i=\ell+1}^m\int_{2B}\big|f_i(y_i)\big|\,dy_i\notag\\
&\times
\sum_{j=1}^\infty\frac{1}{|2^{j+1}B|^m}\int_{(2^{j+1}B)^\ell\backslash(2^{j}B)^\ell}
\big|f_1(y_1)\cdots f_{\ell}(y_\ell)\big|\,dy_1\cdots dy_\ell\notag\\
&\le C\prod_{i=\ell+1}^m\int_{2B}\big|f_i(y_i)\big|\,dy_i\times\sum_{j=1}^\infty\frac{1}{|2^{j+1}B|^m}\prod_{i=1}^{\ell}
\int_{2^{j+1}B\backslash 2^{j}B}\big|f_i(y_i)\big|\,dy_i\notag\\
&\le C\sum_{j=1}^\infty\prod_{i=1}^m\frac{1}{|2^{j+1}B|}\int_{2^{j+1}B}\big|f_i(y_i)\big|\,dy_i,
\end{align}
and we arrived at the expression considered in the previous case. So for any $x\in B$, we also have
\begin{equation*}
\big|T(f^\infty_1,\ldots,f^\infty_\ell,f^0_{\ell+1},\ldots,f^0_m)(x)\big|\le C\prod_{i=1}^m \big\|f_i\big\|_{L^{p_i,\kappa}(w_i)}\cdot
\sum_{j=1}^\infty\nu_{\vec{w}}\big(2^{j+1}B\big)^{{(\kappa-1)}/p}.
\end{equation*}
Therefore, by the inequality (3.2) and the above pointwise inequality, we have
\begin{equation*}
\begin{split}
I^{\alpha_1,\ldots,\alpha_m}&\le\nu_{\vec{w}}(B)^{{(1-\kappa)}/p}
\big|T(f^\infty_1,\ldots,f^\infty_\ell,f^0_{\ell+1},\ldots,f^0_m)(x)\big|\\
&\le C\prod_{i=1}^m \big\|f_i\big\|_{L^{p_i,\kappa}(w_i)}\cdot\sum_{j=1}^\infty
\frac{\nu_{\vec{w}}(B)^{{(1-\kappa)}/p}}{\nu_{\vec{w}}(2^{j+1}B)^{{(1-\kappa)}/p}}\\
\end{split}
\end{equation*}
\begin{equation*}
\begin{split}
&\le C\prod_{i=1}^m \big\|f_i\big\|_{L^{p_i,\kappa}(w_i)}\cdot\sum_{j=1}^\infty
\left(\frac{|B|}{|2^{j+1}B|}\right)^{{\delta(1-\kappa)}/p}\\
&\le C\prod_{i=1}^m \big\|f_i\big\|_{L^{p_i,\kappa}(w_i)}.
\end{split}
\end{equation*}
Combining the above estimates and then taking the supremum over all balls $B\subseteq\mathbb R^n$, we complete the proof of Theorem 1.1.
\end{proof}

\begin{proof}[Proof of Theorem 1.2]
For any ball $B=B(x_0,r_B)\subseteq\mathbb R^n$ and decompose $f_i=f^0_i+f^{\infty}_i$, where $f^0_i=f_i\chi_{2B}$, $i=1,\ldots,m$. Then for any given $\lambda>0$, we can write
\begin{equation*}
\begin{split}
&\nu_{\vec{w}}\big(\big\{x\in B:\big|T(f_1,\ldots,f_m)\big|>\lambda\big\}\big)^{1/p}\\
\le&\nu_{\vec{w}}\big(\big\{x\in B:\big|T(f^0_1,\ldots,f^0_m)\big|>\lambda/{2^m}\big\}\big)^{1/p}
+\sum\nolimits^\prime \nu_{\vec{w}}\big(\big\{x\in B:\big|T(f^{\alpha_1}_1,\ldots,f^{\alpha_m}_m)\big|>\lambda/{2^m}\big\}\big)^{1/p}\\
=&I^0_*+\sum\nolimits^\prime I^{\alpha_1,\ldots,\alpha_m}_*,
\end{split}
\end{equation*}
where each term of $\sum^\prime$ contains at least one $\alpha_i\neq0$. By Lemma 2.4 again, we know that $\nu_{\vec{w}}\in A_{mp}$ with $1\le mp<\infty$.
Applying Theorem B, Lemma 3.1 and Lemma 2.1, we have
\begin{equation*}
\begin{split}
I^0_*&\le\frac{C}{\lambda}\prod_{i=1}^m\left(\int_{2B}|f_i(x)|^{p_i}w_i(x)\,dx\right)^{1/{p_i}}\\
&\le \frac{C\cdot\prod_{i=1}^m w_i(2B)^{\kappa/{p_i}}}{\lambda}\prod_{i=1}^m\big\|f_i\big\|_{L^{p_i,\kappa}(w_i)}\\
&\le \frac{C\cdot\nu_{\vec{w}}(2B)^{\kappa/p}}{\lambda}\prod_{i=1}^m\big\|f_i\big\|_{L^{p_i,\kappa}(w_i)}\\
&\le \frac{C\cdot\nu_{\vec{w}}(B)^{\kappa/p}}{\lambda}\prod_{i=1}^m \big\|f_i\big\|_{L^{p_i,\kappa}(w_i)}.
\end{split}
\end{equation*}
In the proof of Theorem 1.1, we have already showed the following pointwise estimate (see (3.1) and (3.3)).
\begin{equation}
\big|T(f^{\alpha_1}_1,\ldots,f^{\alpha_m}_m)(x)\big|\le C\sum_{j=1}^\infty\prod_{i=1}^m\frac{1}{|2^{j+1}B|}\int_{2^{j+1}B}\big|f_i(y_i)\big|\,dy_i.
\end{equation}
Without loss of generality, we may assume that $p_1=\cdots=p_{\ell}=\min\{p_1,\ldots,p_m\}=1$, and $p_{\ell+1},\ldots,p_m>1$. Using H\"older's inequality, the multiple $A_{\vec{P}}$ condition and Lemma 3.1, we obtain
\begin{equation*}
\begin{split}
\big|T(f^{\alpha_1}_1,\ldots,f^{\alpha_m}_m)(x)\big|&\le
C\sum_{j=1}^\infty\prod_{i=1}^{\ell}\frac{1}{|2^{j+1}B|}\int_{2^{j+1}B}\big|f_i(y_i)\big|\,dy_i\times
\prod_{i=\ell+1}^{m}\frac{1}{|2^{j+1}B|}\int_{2^{j+1}B}\big|f_i(y_i)\big|\,dy_i\\
&\le C\sum_{j=1}^\infty\prod_{i=1}^{\ell}\frac{1}{|2^{j+1}B|}\int_{2^{j+1}B}\big|f_i(y_i)\big|w_i(y_i)\,dy_i
\left(\inf_{y_i\in 2^{j+1}B}w_i(y_i)\right)^{-1}\\
\end{split}
\end{equation*}
\begin{equation*}
\begin{split}
&\times\prod_{i=\ell+1}^{m}\frac{1}{|2^{j+1}B|}
\left(\int_{2^{j+1}B}\big|f_i(y_i)\big|^{p_i}w_i(y_i)\,dy_i\right)^{1/{p_i}}
\left(\int_{2^{j+1}B}w_i(y_i)^{1-p'_i}\,dy_i\right)^{1/{p'_i}}\\
&\le C\prod_{i=1}^m \big\|f_i\big\|_{L^{p_i,\kappa}(w_i)}
\sum_{j=1}^\infty\nu_{\vec{w}}\big(2^{j+1}B\big)^{{(\kappa-1)}/p}.
\end{split}
\end{equation*}
Observe that $\nu_{\vec{w}}\in A_{mp}$ with $1\le mp<\infty$. Thus, it follows from the inequality (3.2) that for any $x\in B$,
\begin{align}
\big|T(f^{\alpha_1}_1,\ldots,f^{\alpha_m}_m)(x)\big|
&= C\prod_{i=1}^m \big\|f_i\big\|_{L^{p_i,\kappa}(w_i)}\cdot\frac{1}{\nu_{\vec{w}}(B)^{{(1-\kappa)}/p}}
\sum_{j=1}^\infty\frac{\nu_{\vec{w}}(B)^{{(1-\kappa)}/p}}{\nu_{\vec{w}}(2^{j+1}B)^{{(1-\kappa)}/p}}\notag\\
&\le C\prod_{i=1}^m \big\|f_i\big\|_{L^{p_i,\kappa}(w_i)}\cdot\frac{1}{\nu_{\vec{w}}(B)^{{(1-\kappa)}/p}}
\sum_{j=1}^\infty\left(\frac{|B|}{|2^{j+1}B|}\right)^{{\delta(1-\kappa)}/p}\notag\\
&\le C\prod_{i=1}^m \big\|f_i\big\|_{L^{p_i,\kappa}(w_i)}\cdot\frac{1}{\nu_{\vec{w}}(B)^{{(1-\kappa)}/p}}.
\end{align}
If $\big\{x\in B:\big|T(f^{\alpha_1}_1,\ldots,f^{\alpha_m}_m)(x)\big|>\lambda/{2^m}\big\}=\O$, then the inequality
\begin{equation*}
I^{\alpha_1,\ldots,\alpha_m}_*\le\frac{C\cdot\nu_{\vec{w}}(B)^{\kappa/p}}{\lambda}\prod_{i=1}^m \big\|f_i\big\|_{L^{p_i,\kappa}(w_i)}
\end{equation*}
holds trivially. Now if instead we suppose that $\big\{x\in B:\big|T(f^{\alpha_1}_1,\ldots,f^{\alpha_m}_m)(x)\big|>\lambda/{2^m}\big\}\neq\O$, then by the pointwise inequality (3.5), we have
\begin{equation*}
\lambda< C\prod_{i=1}^m \big\|f_i\big\|_{L^{p_i,\kappa}(w_i)}\cdot\frac{1}{\nu_{\vec{w}}(B)^{{(1-\kappa)}/p}},
\end{equation*}
which is equivalent to
\begin{equation*}
\nu_{\vec{w}}(B)^{1/p}\le\frac{C\cdot\nu_{\vec{w}}(B)^{\kappa/p}}{\lambda}\prod_{i=1}^m \big\|f_i\big\|_{L^{p_i,\kappa}(w_i)}.
\end{equation*}
Therefore
\begin{equation*}
I^{\alpha_1,\ldots,\alpha_m}_*\le
\nu_{\vec{w}}(B)^{1/p}\le\frac{C\cdot\nu_{\vec{w}}(B)^{\kappa/p}}{\lambda}\prod_{i=1}^m \big\|f_i\big\|_{L^{p_i,\kappa}(w_i)}.
\end{equation*}
Summing up all the above estimates and then taking the supremum over all balls $B\subseteq\mathbb R^n$ and all $\lambda>0$, we complete the proof of Theorem 1.2.
\end{proof}

By using H\"older's inequality, it is easy to check that if each $w_i$ is in $A_{p_i}$, then
\begin{equation*}
\prod_{i=1}^m A_{p_i}\subset A_{\vec{P}}.
\end{equation*}
and this inclusion is strict (see \cite{lerner}). Thus, as direct consequences of Theorems
1.1 and 1.2, we immediately obtain the following

\newtheorem{corollary}[theorem]{Corollary}

\begin{corollary}
Let $m\ge2$ and $T$ be an $m$-linear Calder\'on--Zygmund operator. If $p_1,\ldots,p_m\in(1,\infty)$ and $p\in(0,\infty)$ with $1/p=\sum_{k=1}^m 1/{p_k}$, and $\vec{w}=(w_1,\ldots,w_m)\in \prod_{i=1}^m A_{p_i}$, then for any $0<\kappa<1$, there exists a
constant $C>0$ independent of $\vec{f}=(f_1,\ldots,f_m)$ such that
\begin{equation*}
\big\|T(\vec{f})\big\|_{L^{p,\kappa}(\nu_{\vec{w}})}\le C\prod_{i=1}^m\big\|f_i\big\|_{L^{p_i,\kappa}(w_i)},
\end{equation*}
where $\nu_{\vec{w}}=\prod_{i=1}^m w_i^{p/{p_i}}$.
\end{corollary}

\begin{corollary}
Let $m\ge2$ and $T$ be an $m$-linear Calder\'on--Zygmund operator. If $p_1,\ldots,p_m\in[1,\infty)$, $\min\{p_1,\ldots,p_m\}=1$ and $p\in(0,\infty)$ with $1/p=\sum_{k=1}^m 1/{p_k}$, and $\vec{w}=(w_1,\ldots,w_m)\in \prod_{i=1}^m A_{p_i}$, then for any $0<\kappa<1$, there exists a
constant $C>0$ independent of $\vec{f}=(f_1,\ldots,f_m)$ such that
\begin{equation*}
\big\|T(\vec{f})\big\|_{WL^{p,\kappa}(\nu_{\vec{w}})}\le C\prod_{i=1}^m\big\|f_i\big\|_{L^{p_i,\kappa}(w_i)},
\end{equation*}
where $\nu_{\vec{w}}=\prod_{i=1}^m w_i^{p/{p_i}}$.
\end{corollary}

\section{Proofs of Theorems 1.3 and 1.4}

Following along the same lines as that of Lemma 3.1, we can also show the following result, which plays an important role in our proofs of Theorems 1.3 and 1.4.

\begin{lemma}
Let $m\ge2$, $q_1,\ldots,q_m\in[1,\infty)$ and $q\in(0,\infty)$ with $1/q=\sum_{k=1}^m 1/{q_k}$. Assume that $w^{q_1}_1,\ldots,w^{q_m}_m\in A_\infty$ and $\nu_{\vec{w}}=\prod_{i=1}^m w_i$, then for any ball $B$, there exists a constant $C>0$ such that
\begin{equation*}
\prod_{i=1}^m\left(\int_B w^{q_i}_i(x)\,dx\right)^{q/{q_i}}\le C\int_B\nu_{\vec{w}}(x)^q\,dx.
\end{equation*}
\end{lemma}

\begin{proof}[Proof of Theorem 1.3]
Arguing as in the proof of Theorem 1.1, fix a ball $B=B(x_0,r_B)\subseteq\mathbb R^n$ and decompose $f_i=f^0_i+f^{\infty}_i$, where $f^0_i=f_i\chi_{2B}$, $i=1,\ldots,m$. Since $I_\alpha$ is an $m$-linear operator, then we have
\begin{equation*}
\begin{split}
&\frac{1}{\nu_{\vec{w}}^q(B)^{\kappa/p}}\left(\int_B \big|I_\alpha(f_1,\ldots,f_m)(x)\big|^q\nu_{\vec{w}}(x)^q\,dx\right)^{1/q}\\
\le& \frac{1}{\nu_{\vec{w}}^q(B)^{\kappa/p}}\left(\int_B \big|I_\alpha (f^0_1,\ldots,f^0_m)(x)\big|^q\nu_{\vec{w}}(x)^q\,dx\right)^{1/q}\\
&+\sum\nolimits^\prime\frac{1}{\nu_{\vec{w}}^q(B)^{\kappa/p}}\left(\int_B \big|I_\alpha(f^{\alpha_1}_1,\ldots,f^{\alpha_m}_m)(x)\big|^q\nu_{\vec{w}}(x)^q\,dx\right)^{1/q}\\
=& J^0+\sum\nolimits^\prime J^{\alpha_1,\ldots,\alpha_m},
\end{split}
\end{equation*}
where each term of $\sum^\prime$ contains at least one $\alpha_i\neq0$. In view of Lemma 2.5, we can see that $(\nu_{\vec{w}})^q\in A_{mq}$. Using Theorem C, Lemma 4.1 and Lemma 2.1, we get
\begin{equation*}
\begin{split}
J^0&\le C\cdot
\frac{1}{\nu_{\vec{w}}^q(B)^{\kappa/p}}\prod_{i=1}^m\left(\int_{2B}|f_i(x)|^{p_i}w_i(x)^{p_i}\,dx\right)^{1/{p_i}}\\
&\le C\prod_{i=1}^m\big\|f_i\big\|_{L^{p_i,{\kappa p_iq}/{pq_i}}(w^{p_i}_i,w^{q_i}_i)}\cdot
\frac{\prod_{i=1}^m w^{q_i}_i(2B)^{{\kappa q}/{pq_i}}}{\nu_{\vec{w}}^q(B)^{\kappa/p}}\\
&= C\prod_{i=1}^m\big\|f_i\big\|_{L^{p_i,{\kappa p_iq}/{pq_i}}(w^{p_i}_i,w^{q_i}_i)}\cdot
\frac{\Big(\prod_{i=1}^m w^{q_i}_i(2B)^{q/{q_i}}\Big)^{\kappa/p}}{\nu_{\vec{w}}^q(B)^{\kappa/p}}\\
&\le C\prod_{i=1}^m\big\|f_i\big\|_{L^{p_i,{\kappa p_iq}/{pq_i}}(w^{p_i}_i,w^{q_i}_i)}\cdot
\frac{\nu_{\vec{w}}^q(2B)^{\kappa/p}}{\nu_{\vec{w}}^q(B)^{\kappa/p}}\\
&\le C\prod_{i=1}^m \big\|f_i\big\|_{L^{p_i,{\kappa p_iq}/{pq_i}}(w^{p_i}_i,w^{q_i}_i)}.
\end{split}
\end{equation*}
For the other terms, let us first deal with the case when $\alpha_1=\cdots=\alpha_m=\infty$. By the definition of $I_\alpha$, for any $x\in B$, we obtain
\begin{align}
\big|I_{\alpha}(f^\infty_1,\ldots,f^\infty_m)(x)\big|&= \int_{(\mathbb R^n)^m\backslash(2B)^m}
\frac{|f_1(y_1)\cdots f_m(y_m)|}{(|x-y_1|+\cdots+|x-y_m|)^{mn-\alpha}}dy_1\cdots dy_m\notag\\
&= \sum_{j=1}^\infty\int_{(2^{j+1}B)^m\backslash(2^{j}B)^m}
\frac{|f_1(y_1)\cdots f_m(y_m)|}{(|x-y_1|+\cdots+|x-y_m|)^{mn-\alpha}}dy_1\cdots dy_m\notag\\
&\le C\sum_{j=1}^\infty\prod_{i=1}^m\int_{2^{j+1}B\backslash 2^{j}B}\frac{|f_i(y_i)|}{|x-y_i|^{n-\alpha/m}}dy_i\notag\\
&\le C\sum_{j=1}^\infty\prod_{i=1}^m\frac{1}{|2^{j+1}B|^{1-\alpha/{mn}}}\int_{2^{j+1}B}\big|f_i(y_i)\big|\,dy_i.
\end{align}
Moreover, by using H\"older's inequality, the multiple $A_{\vec{P},q}$ condition and Lemma 4.1, we deduce that
\begin{equation*}
\begin{split}
\big|I_{\alpha}(f^\infty_1,\ldots,f^\infty_m)(x)\big|&\le C\sum_{j=1}^\infty\prod_{i=1}^m\frac{1}{|2^{j+1}B|^{1-\alpha/{mn}}}
\left(\int_{2^{j+1}B}\big|f_i(y_i)\big|^{p_i}w_i(y_i)^{p_i}\,dy_i\right)^{1/{p_i}}\\
&\times\left(\int_{2^{j+1}B}w_i(y_i)^{-p'_i}\,dy_i\right)^{1/{p'_i}}\\
&\le C\sum_{j=1}^\infty\frac{1}{|2^{j+1}B|^{m-\alpha/{n}}}\cdot
\frac{|2^{j+1}B|^{\frac 1q+\sum_{i=1}^m(1-\frac 1{p_i})}}{\nu_{\vec{w}}^q(2^{j+1}B)^{1/q}}\\
\end{split}
\end{equation*}
\begin{equation*}
\begin{split}
&\times\prod_{i=1}^m
\left(\big\|f_i\big\|_{L^{p_i,{\kappa p_iq}/{pq_i}}(w^{p_i}_i,w^{q_i}_i)}w^{q_i}_i\big(2^{j+1}B\big)^{{\kappa q}/{pq_i}}\right)\\
&\le C\prod_{i=1}^m \big\|f_i\big\|_{L^{p_i,{\kappa p_iq}/{pq_i}}(w^{p_i}_i,w^{q_i}_i)}\cdot
\sum_{j=1}^\infty\Bigg[\frac{\Big(\prod_{i=1}^m w^{q_i}_i(2^{j+1}B)^{q/{q_i}}\Big)^{\kappa/p}}{\nu_{\vec{w}}^q(2^{j+1}B)^{1/q}}\Bigg]\\
&\le C\prod_{i=1}^m \big\|f_i\big\|_{L^{p_i,{\kappa p_iq}/{pq_i}}(w^{p_i}_i,w^{q_i}_i)}\cdot
\sum_{j=1}^\infty\nu_{\vec{w}}^q\big(2^{j+1}B\big)^{\kappa/p-1/q}.
\end{split}
\end{equation*}
Since $(\nu_{\vec{w}})^q\in A_{mq}\subset A_\infty$, then it follows immediately from Lemma 2.3 that
\begin{equation}
\frac{\nu_{\vec{w}}^q(B)}{\nu_{\vec{w}}^q(2^{j+1}B)}\le C\left(\frac{|B|}{|2^{j+1}B|}\right)^{\delta^\prime}.
\end{equation}
Hence
\begin{equation*}
\begin{split}
J^{\infty,\ldots,\infty}&\le\nu_{\vec{w}}^q(B)^{1/q-\kappa/p}\big|I_{\alpha}(f^\infty_1,\ldots,f^\infty_m)(x)\big|\\
&\le C\prod_{i=1}^m \big\|f_i\big\|_{L^{p_i,{\kappa p_iq}/{pq_i}}(w^{p_i}_i,w^{q_i}_i)}\cdot\sum_{j=1}^\infty
\frac{\nu_{\vec{w}}^q(B)^{1/q-\kappa/p}}{\nu_{\vec{w}}^q(2^{j+1}B)^{1/q-\kappa/p}}\\
&\le C\prod_{i=1}^m \big\|f_i\big\|_{L^{p_i,{\kappa p_iq}/{pq_i}}(w^{p_i}_i,w^{q_i}_i)}\cdot\sum_{j=1}^\infty
\left(\frac{|B|}{|2^{j+1}B|}\right)^{\delta^\prime(1/q-\kappa/p)}\\
&\le C\prod_{i=1}^m \big\|f_i\big\|_{L^{p_i,{\kappa p_iq}/{pq_i}}(w^{p_i}_i,w^{q_i}_i)},
\end{split}
\end{equation*}
where in the last inequality we have used the fact that $0<\kappa<p/q$ and $\delta^\prime>0$. We now consider the case where exactly $\ell$ of the $\alpha_i$ are $\infty$ for some $1\le\ell<m$. We
only give the arguments for one of these cases. The rest are similar and can easily be obtained from the arguments below by permuting the indices. Using the definition of $I_\alpha$ again, we can see that for any $x\in B$,
\begin{align}
\big|I_{\alpha}(f^\infty_1,\ldots,f^\infty_\ell,f^0_{\ell+1},\ldots,f^0_m)(x)\big|&=\int_{(\mathbb R^n)^{\ell}\backslash(2B)^{\ell}}\int_{(2B)^{m-\ell}}\frac{|f_1(y_1)\cdots f_m(y_m)|}{(|x-y_1|+\cdots+|x-y_m|)^{mn-\alpha}}dy_1\cdots dy_m\notag\\
&\le C\prod_{i=\ell+1}^m\int_{2B}\big|f_i(y_i)\big|\,dy_i\notag\\
&\times
\sum_{j=1}^\infty\frac{1}{|2^{j+1}B|^{m-\alpha/n}}\int_{(2^{j+1}B)^\ell\backslash(2^{j}B)^\ell}
\big|f_1(y_1)\cdots f_{\ell}(y_\ell)\big|\,dy_1\cdots dy_\ell\notag\\
&\le C\prod_{i=\ell+1}^m\int_{2B}\big|f_i(y_i)\big|\,dy_i\times
\sum_{j=1}^\infty\frac{1}{|2^{j+1}B|^{m-\alpha/n}}\prod_{i=1}^{\ell}
\int_{2^{j+1}B\backslash 2^{j}B}\big|f_i(y_i)\big|\,dy_i\notag\\
&\le C\sum_{j=1}^\infty\prod_{i=1}^m\frac{1}{|2^{j+1}B|^{1-\alpha/{mn}}}\int_{2^{j+1}B}\big|f_i(y_i)\big|\,dy_i,
\end{align}
and we arrived at the expression considered in the previous case. Thus, for any $x\in B$, we also have
\begin{equation*}
\big|I_{\alpha}(f^\infty_1,\ldots,f^\infty_\ell,f^0_{\ell+1},\ldots,f^0_m)(x)\big|\le C\prod_{i=1}^m \big\|f_i\big\|_{L^{p_i,{\kappa p_iq}/{pq_i}}(w^{p_i}_i,w^{q_i}_i)}\cdot
\sum_{j=1}^\infty\nu_{\vec{w}}^q\big(2^{j+1}B\big)^{\kappa/p-1/q}.
\end{equation*}
Therefore, by the inequality (4.2) and the above pointwise inequality, we obtain
\begin{equation*}
\begin{split}
J^{\alpha_1,\ldots,\alpha_m}&\le\nu_{\vec{w}}^q(B)^{1/q-\kappa/p}
\big|I_{\alpha}(f^\infty_1,\ldots,f^\infty_\ell,f^0_{\ell+1},\ldots,f^0_m)(x)\big|\\
&\le C\prod_{i=1}^m \big\|f_i\big\|_{L^{p_i,{\kappa p_iq}/{pq_i}}(w^{p_i}_i,w^{q_i}_i)}\cdot\sum_{j=1}^\infty
\frac{\nu_{\vec{w}}^q(B)^{1/q-\kappa/p}}{\nu_{\vec{w}}^q(2^{j+1}B)^{1/q-\kappa/p}}\\
&\le C\prod_{i=1}^m \big\|f_i\big\|_{L^{p_i,{\kappa p_iq}/{pq_i}}(w^{p_i}_i,w^{q_i}_i)}\cdot\sum_{j=1}^\infty
\left(\frac{|B|}{|2^{j+1}B|}\right)^{\delta^\prime(1/q-\kappa/p)}\\
&\le C\prod_{i=1}^m \big\|f_i\big\|_{L^{p_i,{\kappa p_iq}/{pq_i}}(w^{p_i}_i,w^{q_i}_i)}.
\end{split}
\end{equation*}
Summarizing the estimates derived above and then taking the supremum over all balls $B\subseteq\mathbb R^n$, we finish the proof of Theorem 1.3.
\end{proof}

\begin{proof}[Proof of Theorem 1.4]
As before, fix a ball $B=B(x_0,r_B)\subseteq\mathbb R^n$ and split $f_i$ into $f_i=f^0_i+f^{\infty}_i$, where $f^0_i=f_i\chi_{2B}$, $i=1,\ldots,m$. Then for each fixed $\lambda>0$, we can write
\begin{equation*}
\begin{split}
&\nu_{\vec{w}}^q\big(\big\{x\in B:\big|I_{\alpha}(f_1,\ldots,f_m)\big|>\lambda\big\}\big)^{1/q}\\
\le&\nu_{\vec{w}}^q\big(\big\{x\in B:\big|I_{\alpha}(f^0_1,\ldots,f^0_m)\big|>\lambda/{2^m}\big\}\big)^{1/q}
+\sum\nolimits^\prime \nu_{\vec{w}}^q\big(\big\{x\in B:\big|I_{\alpha}(f^{\alpha_1}_1,\ldots,f^{\alpha_m}_m)\big|>\lambda/{2^m}\big\}\big)^{1/q}\\
=&J^0_*+\sum\nolimits^\prime J^{\alpha_1,\ldots,\alpha_m}_*,
\end{split}
\end{equation*}
where each term of $\sum^\prime$ contains at least one $\alpha_i\neq0$. By Lemma 2.5 again, we know that $(\nu_{\vec{w}})^q\in A_{mq}$ with $1<mq<\infty$. Using Theorem D, Lemma 4.1 and Lemma 2.1, we have
\begin{equation*}
\begin{split}
J^0_*&\le\frac{C}{\lambda}\prod_{i=1}^m\left(\int_{2B}|f_i(x)|^{p_i}w_i(x)^{p_i}\,dx\right)^{1/{p_i}}\\
&\le \frac{C\cdot\prod_{i=1}^m w^{q_i}_i(2B)^{{\kappa q}/{pq_i}}}{\lambda}\prod_{i=1}^m\big\|f_i\big\|_{L^{p_i,{\kappa p_iq}/{pq_i}}(w^{p_i}_i,w^{q_i}_i)}\\
&\le \frac{C\cdot\nu_{\vec{w}}^q(2B)^{\kappa/p}}{\lambda}\prod_{i=1}^m\big\|f_i\big\|_{L^{p_i,{\kappa p_iq}/{pq_i}}(w^{p_i}_i,w^{q_i}_i)}\\
&\le \frac{C\cdot\nu_{\vec{w}}^q(B)^{\kappa/p}}{\lambda}\prod_{i=1}^m\big\|f_i\big\|_{L^{p_i,{\kappa p_iq}/{pq_i}}(w^{p_i}_i,w^{q_i}_i)}.
\end{split}
\end{equation*}
In the proof of Theorem 1.3, we have already proved the following pointwise estimate (see (4.1) and (4.3)).
\begin{equation}
\big|I_{\alpha}(f^{\alpha_1}_1,\ldots,f^{\alpha_m}_m)(x)\big|\le C\sum_{j=1}^\infty\prod_{i=1}^m\frac{1}{|2^{j+1}B|^{1-\alpha/{mn}}}\int_{2^{j+1}B}\big|f_i(y_i)\big|\,dy_i.
\end{equation}
Without loss of generality, we may assume that $p_1=\cdots=p_{\ell}=\min\{p_1,\ldots,p_m\}=1$, and $p_{\ell+1},\ldots,p_m>1$. By using H\"older's inequality, the multiple $A_{\vec{P},q}$ condition and Lemma 4.1, we obtain
\begin{equation*}
\begin{split}
\big|I_{\alpha}(f^{\alpha_1}_1,\ldots,f^{\alpha_m}_m)(x)\big|&\le
C\sum_{j=1}^\infty\prod_{i=1}^{\ell}\frac{1}{|2^{j+1}B|^{1-\alpha/{mn}}}\int_{2^{j+1}B}\big|f_i(y_i)\big|\,dy_i\\
&\times\prod_{i=\ell+1}^{m}\frac{1}{|2^{j+1}B|^{1-\alpha/{mn}}}\int_{2^{j+1}B}\big|f_i(y_i)\big|\,dy_i\\
&\le C\sum_{j=1}^\infty\prod_{i=1}^{\ell}
\frac{1}{|2^{j+1}B|^{1-\alpha/{mn}}}\int_{2^{j+1}B}\big|f_i(y_i)\big|w_i(y_i)\,dy_i
\left(\inf_{y_i\in 2^{j+1}B}w_i(y_i)\right)^{-1}\\
&\times\prod_{i=\ell+1}^{m}\frac{1}{|2^{j+1}B|^{1-\alpha/{mn}}}
\left(\int_{2^{j+1}B}\big|f_i(y_i)\big|^{p_i}w_i(y_i)^{p_i}\,dy_i\right)^{1/{p_i}}
\left(\int_{2^{j+1}B}w_i(y_i)^{-p'_i}\,dy_i\right)^{1/{p'_i}}\\
&\le C\prod_{i=1}^m \big\|f_i\big\|_{L^{p_i,{\kappa p_iq}/{pq_i}}(w^{p_i}_i,w^{q_i}_i)}\cdot
\sum_{j=1}^\infty\nu_{\vec{w}}^q\big(2^{j+1}B\big)^{\kappa/p-1/q}.
\end{split}
\end{equation*}
Note that $(\nu_{\vec{w}})^q\in A_{mq}$ with $1<mq<\infty$. Hence, it follows from the inequality (4.2) that for any $x\in B$,
\begin{align}
\big|I_{\alpha}(f^{\alpha_1}_1,\ldots,f^{\alpha_m}_m)(x)\big|
&= C\prod_{i=1}^m \big\|f_i\big\|_{L^{p_i,{\kappa p_iq}/{pq_i}}(w^{p_i}_i,w^{q_i}_i)}\cdot
\frac{1}{\nu_{\vec{w}}^q(B)^{1/q-\kappa/p}}
\sum_{j=1}^\infty\frac{\nu_{\vec{w}}^q(B)^{1/q-\kappa/p}}{\nu_{\vec{w}}^q(2^{j+1}B)^{1/q-\kappa/p}}\notag\\
&\le C\prod_{i=1}^m \big\|f_i\big\|_{L^{p_i,{\kappa p_iq}/{pq_i}}(w^{p_i}_i,w^{q_i}_i)}\cdot
\frac{1}{\nu_{\vec{w}}^q(B)^{1/q-\kappa/p}}\sum_{j=1}^\infty
\left(\frac{|B|}{|2^{j+1}B|}\right)^{\delta^\prime(1/q-\kappa/p)}\notag\\
&\le C\prod_{i=1}^m \big\|f_i\big\|_{L^{p_i,{\kappa p_iq}/{pq_i}}(w^{p_i}_i,w^{q_i}_i)}\cdot
\frac{1}{\nu_{\vec{w}}^q(B)^{1/q-\kappa/p}}.
\end{align}
If $\big\{x\in B:\big|I_{\alpha}(f^{\alpha_1}_1,\ldots,f^{\alpha_m}_m)(x)\big|>\lambda/{2^m}\big\}=\O$, then the inequality
\begin{equation*}
J^{\alpha_1,\ldots,\alpha_m}_*\le\frac{C\cdot\nu_{\vec{w}}^q(B)^{\kappa/p}}{\lambda}\prod_{i=1}^m \big\|f_i\big\|_{L^{p_i,{\kappa p_iq}/{pq_i}}(w^{p_i}_i,w^{q_i}_i)}
\end{equation*}
holds trivially. Now if instead we assume that $\big\{x\in B:\big|I_{\alpha}(f^{\alpha_1}_1,\ldots,f^{\alpha_m}_m)(x)\big|>\lambda/{2^m}\big\}\neq\O$, then by the pointwise inequality (4.5), we get
\begin{equation*}
\lambda<C\prod_{i=1}^m \big\|f_i\big\|_{L^{p_i,{\kappa p_iq}/{pq_i}}(w^{p_i}_i,w^{q_i}_i)}\cdot
\frac{1}{\nu_{\vec{w}}^q(B)^{1/q-\kappa/p}},
\end{equation*}
which in turn gives that
\begin{equation*}
\nu_{\vec{w}}^q(B)^{1/q}\le\frac{C\cdot\nu_{\vec{w}}^q(B)^{\kappa/p}}{\lambda}\prod_{i=1}^m \big\|f_i\big\|_{L^{p_i,{\kappa p_iq}/{pq_i}}(w^{p_i}_i,w^{q_i}_i)}.
\end{equation*}
Therefore
\begin{equation*}
J^{\alpha_1,\ldots,\alpha_m}_*\le\nu_{\vec{w}}^q(B)^{1/q}
\le\frac{C\cdot\nu_{\vec{w}}^q(B)^{\kappa/p}}{\lambda}\prod_{i=1}^m \big\|f_i\big\|_{L^{p_i,{\kappa p_iq}/{pq_i}}(w^{p_i}_i,w^{q_i}_i)}.
\end{equation*}
Collecting all the above estimates and then taking the supremum over all balls $B\subseteq\mathbb R^n$ and all $\lambda>0$, we conclude the proof of Theorem 1.4.
\end{proof}
By using H\"older's inequality, it is easy to verify that if $1\le p_i<q_i$, $1/q=\sum_{k=1}^m 1/{q_k}$ and each $w_i$ is in $A_{p_i,q_i}$, then we have
\begin{equation*}
\prod_{i=1}^m A_{p_i,q_i}\subset A_{\vec{P},q}.
\end{equation*}
and this inclusion is strict (see \cite{moen}).Also recall that $w\in A_{p,q}$ if and only if $w^q\in A_{1+q/{p'}}\subset A_\infty$(see \cite{muckenhoupt2}). Thus, as straightforward consequences of Theorems
1.3 and 1.4, we finally obtain the following

\begin{corollary}
Let $m\ge2$, $0<\alpha<mn$ and $I_\alpha$ be an $m$-linear fractional integral operator. If $p_1,\ldots,p_m\in(1,\infty)$, $1/p=\sum_{k=1}^m 1/{p_k}$, $1/{q_k}=1/{p_k}-\alpha/{mn}$ and $1/q=\sum_{k=1}^m 1/{q_k}=1/p-\alpha/n$, and $\vec{w}=(w_1,\ldots,w_m)\in \prod_{i=1}^m A_{p_i,q_i}$, then for any $0<\kappa<p/q$, there exists a constant $C>0$ independent of $\vec{f}=(f_1,\ldots,f_m)$ such that
\begin{equation*}
\big\|I_\alpha(\vec{f})\big\|_{L^{q,{\kappa q}/p}((\nu_{\vec{w}})^q)}\le C\prod_{i=1}^m
\big\|f_i\big\|_{L^{p_i,{\kappa p_iq}/{pq_i}}(w^{p_i}_i,w^{q_i}_i)},
\end{equation*}
where $\nu_{\vec{w}}=\prod_{i=1}^m w_i$.
\end{corollary}

\begin{corollary}
Let $m\ge2$, $0<\alpha<mn$ and $I_\alpha$ be an $m$-linear fractional integral operator. If $p_1,\ldots,p_m\in[1,\infty)$, $\min\{p_1,\ldots,p_m\}=1$, $1/p=\sum_{k=1}^m 1/{p_k}$, $1/{q_k}=1/{p_k}-\alpha/{mn}$ and $1/q=\sum_{k=1}^m 1/{q_k}=1/p-\alpha/n$, and $\vec{w}=(w_1,\ldots,w_m)\in \prod_{i=1}^m A_{p_i,q_i}$, then for any $0<\kappa<p/q$, there exists a constant $C>0$ independent of $\vec{f}=(f_1,\ldots,f_m)$ such that
\begin{equation*}
\big\|I_\alpha(\vec{f})\big\|_{WL^{q,{\kappa q}/p}((\nu_{\vec{w}})^q)}\le C\prod_{i=1}^m
\big\|f_i\big\|_{L^{p_i,{\kappa p_iq}/{pq_i}}(w^{p_i}_i,w^{q_i}_i)},
\end{equation*}
where $\nu_{\vec{w}}=\prod_{i=1}^m w_i$.
\end{corollary}


\begin{thebibliography}{99}

\bibitem{adams} D. R. Adams, A note on Riesz potentials, Duke Math. J, \textbf{42}(1975), 765--778.
\bibitem{chen} X. Chen and Q. Y. Xue, Weighted estimates for a class of multilinear fractional type operators, J. Math. Anal. Appl, \textbf{362}(2010), 355--373.
\bibitem{chiarenza} F. Chiarenza and M. Frasca, Morrey spaces and Hardy--Littlewood maximal function, Rend. Math. Appl, \textbf{7}(1987), 273--279.
\bibitem{coifman} R. R. Coifman and Y. Meyer, On commutators of singular integrals and bilinear singular integrals,      Trans. Amer. Math. Soc, \textbf{212}(1975), 315--331.
\bibitem{duoand} J. Duoandikoetxea, Fourier Analysis, American Mathematical Society, Providence, Rhode Island, 2000.
\bibitem{fan} D. S. Fan, S. Z. Lu and D. C. Yang, Regularity in Morrey spaces of strong solutions to nondivergence
    elliptic equations with VMO coefficients, Georgian Math. J, \textbf{5}(1998), 425--440.
\bibitem{fazio1} G. Di Fazio and M. A. Ragusa, Interior estimates in Morrey spaces for strong solutions to
    nondivergence form equations with discontinuous coefficients, J. Funct. Anal, \textbf{112}(1993), 241--256.
\bibitem{fazio2} G. Di Fazio, D. K. Palagachev and M. A. Ragusa, Global Morrey regularity of strong solutions to the
    Dirichlet problem for elliptic equations with discontinuous coefficients, J. Funct. Anal, \textbf{166}(1999),
    179--196.
\bibitem{garcia} J. Garcia-Cuerva and J. L. Rubio de Francia, Weighted Norm Inequalities and Related Topics, North-Holland, Amsterdam, 1985.
\bibitem{grafakos1} L. Grafakos, On multilinear fractional integrals, Studia Math., \textbf{102}(1992), 49--56.
\bibitem{grafakos5} L. Grafakos and N. Kalton, Multilinear Calder\'on--Zygmund operators on Hardy spaces, Collect. Math., \textbf{52}(2001), 169--179.
\bibitem{grafakos2} L. Grafakos and R. H. Torres, Multilinear Calder\'on--Zygmund theory, Adv. Math., \textbf{165}(2002), 124--164.
\bibitem{grafakos3} L. Grafakos and R. H. Torres, Maximal operator and weighted norm inequalities for multilinear singular integrals, Indiana Univ. Math. J., \textbf{51}(2002), 1261--1276.
\bibitem{grafakos4} L. Grafakos and R. H. Torres, On multilinear singular integrals of Calder\'on--Zygmund type, Publ. Mat., (2002), Vol. Extra, 57--91.
\bibitem{hu} G. E. Hu and Y. Meng, Multilinear Calder\'on--Zygmund operator on products of Hardy spaces, Acta Math. Sinica (Engl. Ser), \textbf{28}(2012), 281--294.
\bibitem{Iida1} T. Iida, E. Sato, Y. Sawano and H. Tanaka, Weighted norm inequalities for multilinear fractional operators on Morrey spaces, Studia Math., \textbf{205}(2011), 139--170.
\bibitem{Iida2} T. Iida, E. Sato, Y. Sawano and H. Tanaka, Multilinear fractional integrals on Morrey spaces, Acta Math. Sinica (Engl. Ser), \textbf{28}(2012), 1375--1384.
\bibitem{Iida3} T. Iida, E. Sato, Y. Sawano and H. Tanaka, Sharp bounds for multilinear fractional integral operators on Morrey type spaces, Positivity, \textbf{16}(2012), 339--358.
\bibitem{kenig} C. E. Kenig and E. M. Stein, Multilinear estimates and fractional integration, Math. Res. Lett., \textbf{6}(1999), 1--15.
\bibitem{komori} Y. Komori and S. Shirai, Weighted Morrey spaces and a singular integral operator, Math. Nachr, \textbf{282}(2009), 219--231.
\bibitem{lerner} A. K. Lerner, S. Ombrosi, C. P\'erez, R. H. Torres and R. Trujillo-Gonz\'alez, New maximal functions and multiple weights for the multilinear Calder\'on--Zygmund theory, Adv. Math., \textbf{220}(2009), 1222--1264.
\bibitem{li1} W. J. Li, Q. Y. Xue and K. Yabuta, Maximal operator for multilinear Calder\'on--Zygmund singular integral operators on weighted Hardy spaces, J. Math. Anal. Appl, \textbf{373}(2011), 384--392.
\bibitem{li2} W. J. Li, Q. Y. Xue and K. Yabuta, Multilinear Calder\'on--Zygmund operators on weighted Hardy spaces, Studia Math., \textbf{199}(2010), 1--16.
\bibitem{moen} K. Moen, Weighted inequalities for multilinear fractional integral operators, Collect. Math., \textbf{60}(2009), 213--238.
\bibitem{morrey} C. B. Morrey, On the solutions of quasi-linear elliptic partial differential equations, Trans. Amer. Math. Soc, \textbf{43}(1938), 126--166.
\bibitem{muckenhoupt1} B. Muckenhoupt, Weighted norm inequalities for the Hardy maximal function, Trans. Amer. Math. Soc, \textbf{165}(1972), 207--226.
\bibitem{muckenhoupt2} B. Muckenhoupt and R. L. Wheeden, Weighted norm inequalities for fractional integrals, Trans. Amer. Math. Soc, \textbf{192}(1974), 261--274.
\bibitem{peetre} J. Peetre, On the theory of $\mathcal L_{p,\lambda}$ spaces, J. Funct. Anal, \textbf{4}(1969),
    71--87.
\bibitem{pradolini} G. Pradolini, Weighted inequalities and pointwise estimates for the multilinear fractional integral and maximal operators, J. Math. Anal. Appl, \textbf{367}(2010), 640--656.
\bibitem{tang} L. Tang, Endpoint estimates for multilinear fractional integrals, J. Aust. Math. Soc, \textbf{84}(2008), 419--429.
\bibitem{wang1} H. Wang, Some estimates for commutators of Calder\'on--Zygmund operators on the weighted Morrey spaces, Sci. Sin. Math, \textbf{42}(2012), 31--45.
\bibitem{wang2} H. Wang, The boundedness of some operators with rough kernel on the weighted Morrey spaces, Acta Math. Sinica (Chin. Ser), \textbf{55}(2012), 589--600.
\bibitem{wang3} H. Wang, Intrinsic square functions on the weighted Morrey spaces, J. Math. Anal. Appl, \textbf{396}(2012), 302--314.
\bibitem{wang4} H. Wang, Boundedness of fractional integral operators with rough kernels on weighted Morrey spaces, Acta Math. Sinica (Chin. Ser), \textbf{56}(2013), 175--186.
\bibitem{wang5} H. Wang, Some estimates for commutators of fractional integral operators on weighted Morrey spaces, Acta Math. Sinica (Chin. Ser), to appear.
\bibitem{wang8} H. Wang, Some estimates for commutators of fractional integrals associated to operators with Gaussian kernel bounds on weighted Morrey spaces, Anal. Theory Appl, to appear.
\bibitem{wang6} H. Wang, Weak type estimates for intrinsic square functions on the weighted Morrey spaces, preprint, 2012.
\bibitem{wang7} H. Wang and H. P. Liu, Some estimates for Bochner--Riesz operators on the weighted Morrey spaces, Acta Math. Sinica (Chin. Ser), \textbf{55}(2012), 551--560.

\end{thebibliography}
\end{document}